\documentclass[reqno]{amsart}
\usepackage{hyperref}
\newcommand{\N}{{\mathbb{N}}}
\newcommand{\R}{{\mathbb{R}}}
\newcommand{\M}{{\mathbb{M}}}
\begin{document}
\title[Pontryagin principle]{Pontryagin principle for a Mayer problem governed by a delay functional differential equation}
\author[Blot  Kon\'e]{Jo\"el Blot and mamadou I. Kon\'e}
\address{Jo\"el Blot: Laboratoire SAMM UE 4543, \newline
Universit\'e Paris 1 Panth\'eon-Sorbonne, centre P.M.F.,\newline
90 rue de Tolbiac, 75634 Paris cedex 13, France.}
\email{blot@univ-paris1.fr}
\address{Mamadou I. Kon\'e: Laboratoire SAMM UE 4543, \newline
Universit\'e Paris 1 Panth\'eon-Sorbonne, centre P.M.F.,\newline
90 rue de Tolbiac, 75634 Paris cedex 13, France.}
\email{mamadou.kone@malix.univ-paris1.fr}
\date{November 9, 2015}
\begin{abstract}
We establish Pontryagin principles for a Mayer's optimal control problem governed by a functional differential equation. The control functions are piecewise continuous and the state functions are piecewise continuously differentiable. To do that, we follow the method created by Philippe Michel for systems governed by ordinary differential equations, and we use properties of the resolvent of a linear functional differential equation.
\end{abstract}

\maketitle
\numberwithin{equation}{section}
\newtheorem{theorem}{Theorem}[section]
\newtheorem{lemma}[theorem]{Lemma}
\newtheorem{example}[theorem]{Example}
\newtheorem{remark}[theorem]{Remark}
\newtheorem{definition}[theorem]{Definition}
\newtheorem{proposition}[theorem]{Proposition}
\noindent
{\bf Key words:} optimal control; Pontryagin principle; functional differential equation.\\
{\bf MSC2010-AMS:} 49J21, 49K21, 34K09.
\section{Introduction}
We consider the following problem of optimal control. It is called a problem of Mayer since its criterion takes into account only the final value of the state; it is governed by a functional differential equation in presence of terminal constraints.
\[
\left.
\begin{array}{rl}
{\rm Maximize} & J(x,u) := g^0(x(T))\\
{\rm when} & x \in C^0([-r,T], \R^n), x \in PC^1([0,T], \R^n)\\
\null & u \in PC^0([0,T], U)\\
\null &\forall t \in [0,T] \setminus F, \;\; x'(t) = f(t, x_t, u(t))\\
\null &x_0 = \phi \\
\null &\forall j = 1,...,n_i, \; g^j(x(T)) \geq 0\\
\null & \forall j= n_i+1,...,n_i+n_e, \; g^j(x(T)) = 0.
\end{array}
\right\} ({\mathfrak M}) 
\]
where the $g^j :  \R^n \rightarrow \R$ are mappings, the state variable $x$ is a piecewise continuously differentiable function (see Section 2), the control variable $u$ is a piecewise continuous function (see Section 2), $U$ is a nonempty subset of $\R^d$, $F$ denotes a finite subset of $[0,T]$ (not a priori fixed), $x_t(\theta) := x(t + \theta)$ when $\theta \in [-r,0]$, $\phi$ is fixed continuous function from $[-r,0]$ into $\R^n$. The only assumptions that we do on this problem are the following ones.
\begin{equation}\label{eq11}
\left.
\begin{array}{l}
f \in C^0([0,T] \times C^0([-r,0], \R^n) \times U, \R^n)\\
\forall (t,\phi,\xi) \in [0,T] \times C^0([-r,0], \R^n) \times U, D_2f(t,\phi, \xi) \; {\rm exists}\\
D_2f \in C^0([0,T] \times C^0([-r,0], \R^n) \times U, {\mathfrak L}( C^0([-r,0], \R^n), \R^n))\\
D_2f \; {\rm is} \; {\rm a} \; {\rm bounded} \; {\rm operator}.
\end{array}
\right\}
\end{equation}
where $C^0(X,Y)$ denotes the space of the continuous mappings from $X$ into $Y$, $D_2f(t,\phi, \xi)$ denotes the partial Fr\'echet differential of $f$ with respect to its second variable, and ${\mathfrak L}(E,E_1)$ is the vector space of the continuous linear mappings from $E$ into $E_1$ when $E$ and $E_1$ are normed vector spaces.
\begin{equation}\label{eq12}
\forall j \in \{ 0,...,n_i + n_e \}, g^j \in C^1(\R^n, \R).
\end{equation}
where $C^1$ means continuously Fr\'echet differentiable.
\vskip1mm
To establish a Pontryagin principle for problem (${\mathfrak M}$) under assumptions which are so light as possible, we follow the method created by Philippe Michel in \cite{Mi} for systems governed by ordinary differential equations. This method is also used in \cite{ATF} for Bolza problems. We can say that this work is an essay to generalize the method of Michel to the setting of systems governed by functional differential equations.
\vskip1mm
On the question of the resolvent of (nonautonomous) linear functional differential equations, the difference between the results that we use (issued from \cite{BK1}) and the results of Banks \cite{Ba} is the choice of the class of solutions: we use continuously differentiable and piecewise continuously  solutions with a continuous vector field and Banks uses absolutely continuous solutions without the continuity of the vector field. On the problem of optimal control, the difference between our setting and the setting of Banks \cite{Ba1} (except that Banks considers a Lagrange problem) is that we use piecewise continuously differentiable state variables and piecewise continuous control variables (as \cite{Mi} and \cite{ATF}) and banks uses absolutely continuous state variables and bounded measurable control variables. 
\vskip1mm
Now we describe the contents of the paper. In Section 2 we specify the notation. In Section 3 we recall some precise properties of the resolvent of linear functional differential equations in the framework of piecewise continuous functions. In Section 4 we give the statement of a Pontryagin principle. In Section 5 we give a proof of this Pontryagin principle.
\section{Notation}
\vskip1mm
$\M_n(\R)$ denotes the space of the real $n \times n$ matrices. $\Vert \cdot \Vert_{\mathfrak L}$ denotes the norm of the linear continuous operators.
\vskip1mm
When $a < b$ are two real numbers,  $C^0_R([a,b], \R^n)$ (respectively $C^0_L([a,b], \R^n)$) is the space of the right-continuous (respectively left-continuous) functions  from $[a,b]$ into $\R^n$, and $C^1([a,b], \R^n)$ is the space of the continuously differentiable functions from $[a,b]$ into $\R^n$.
\vskip1mm
$BV([a,b], \R^n)$ is the space of the bounded variation functions from $[a,b]$ into $\R^n$. When $g \in BV([a,b], \R^n)$ the variation of $g$ on $[a,b]$ is denoted by $V_a^b(g)$. We set $NBV([a,b], \R^n) := \{ g \in BV([a,b], \R^n) \cap C^0_L([a,b], \R^n) : g(a) = 0 \}$.  When $g \in NBV([a,b], \R^n)$, $\Vert g \Vert_{BV} := V_a^b(g)$ defines a norm on $NBV([a,b], \R^n)$. 
\vskip1mm
$AC([a,b], \R^n)$ is the space of the absolutely continuous functions $[a,b]$ into $\R^n$.
\vskip1mm
Let $g : [a,b] \rightarrow \R^n$ be a function, and $t \in [a,b)$ (respectively $(a,b]$) when it exists the right-hand limit (respectively the left-hand limit) of $g$ at $t$ is $g(t+) := \lim_{s \rightarrow t, s >t} g(s)$ (respectively $g(t-) := \lim_{s \rightarrow t, s <t} g(s)$).
\vskip1mm
Let $f : [a,b] \times [c,d] \rightarrow \R^n$ be a mapping, and let $(t,s) \in [a,b) \times [c,d]$ (respectively $(a,b] \times [c,d]$). When it exists the right-partial derivative (respectively left-partial derivative)with respect to the first variable of $f$ at $(t,s)$ is denoted by $\frac{\partial f(t,s)}{\partial t+}$ (respectively $\frac{\partial f(t,s)}{\partial t-}$). 
\vskip1mm
A function $g : [a,b] \rightarrow \R^n$ is called piecewise continuous when it is continuous or when there exists a finite list of points, $t_0=a < t_1< ...< t_p< t_{p+1} = b$ such that $g$ is continuous at each $t \in [a,b] \setminus \{ t_k : k \in \{0, ...,p+1 \} \}$ and such that, for all $k \in \{0,...,p \}$, $g(t_k+)$ exists, and for all $k \in \{ 1,...,p+1 \}$, $g(t_k-)$ exists. We denote by $PC^0([a,b], \R^n)$ the  space of the piecewise continuous functions from $[a,b]$ into $\R^n$. $\Vert g \Vert_{\infty} := \sup \{ \Vert g(t) \Vert : t \in [a,b] \}$ defines a norm on $PC^0([a,b], \R^n)$; endowed with this norm, $PC^0([a,b], \R^n)$ is not complete. When $g \in PC^0([a,b], \R^n)$ we denote by $N_g$ the set points $ t \in [a,b]$ where $g$ is not continuous at $t$. When we fix a finite subset $\pi \subset [a,b]$, we set $PC^0_{\pi}([a,b], \R^n) := \{ g \in PC^0([a,b], \R^n) : N_g \subset \pi \}$. Endowed with $\Vert . \Vert_{\infty}$, $PC^0_{\pi}([a,b], \R^n)$ is a Banach space.
\vskip1mm
A function $g : [a,b] \rightarrow \R^n$ is called piecewise-$C^1$ when $g$ is $C^1$ on $[a,b]$ or when $g \in C^0([a,b], \R^n)$ and there exists a finite list $t_0=a < t_1 < ... < t_p < t_{p+1} = b$ such that $g$ is $C^1$ on $[t_k, t_{k+1}]$ for all $k \in \{ 0, ..., p \}$ and such that, for all $k \in \{ 0,...,p \}$, $g'(t_k +)$ exists and, for all $k \in \{1, ..., p+1 \}$, $g'(t_k -)$ exists. We denote by $PC^1([a,b], \R^n)$  the space of the piecewise-$C^1$  functions from $[a,b]$ into $\R^n$. When $g \in PC^1([a,b], \R^n)$ we denote by $N_{g'}$ the set of the $t \in [a,b]$ such that $g'$ is not continuous at $t$. When we fix a finite subset $\pi \subset [a,b]$, we set $PC^1_{\pi}([a,b], \R^n) := \{ g \in PC^1([a,b], \R^n) : N_{g'} \subset \pi \}$.
\vskip1mm
 In a normed space $E$, when $x \in E$ and $r \in (0, + \infty)$ we set $\overline{B}(x,r) := \{ z \in E : \Vert z -x \Vert \leq r \}$.
\section{Linear functional differential equations}
\subsection{The continuous time framework}
We consider a mapping $L : [0,T] \rightarrow {\mathfrak L}(C^0([-r, 0], \R^n), \R^n)$ which satisfies the following condition.
\begin{equation}\label{eq31}
L \in C^0([0,T], {\mathfrak L}(C^0([-r, 0], \R^n), \R^n)).
\end{equation}
From $L$ and a function $\phi \in C^0([-r, 0], \R^n)$, when $\sigma \in [0,T]$, we consider the following linear functional differential equation under an initial condition.
\begin{equation}\label{eq32}
x'(t) = L(t)x_t, \;\; x_{\sigma} = \phi.
\end{equation}
When moreover $h \in C^0([0,T], \R^n)$ we consider the nonhomogeneous following problem.
\begin{equation}\label{eq33}
x'(t) = L(t)x_t + h(t), \;\; x_{\sigma} = \phi.
\end{equation}
A solution of one of these problems is a function $x \in C^0([-r, T], \R^n)$ which is of class $C^1$ on $[0,T]$ and whom the derative satisfies the equation at each point of $[0,T]$. \\
The only difference between Theorem 4.1 of \cite{BK1} and the following result is the choice of $\eta(t, -r) = 0$ instead of $\eta(t,0) = 0$..
\begin{proposition}\label{prop31}
Under \eqref{eq31} there exists a mapping $\eta : [0,T] \times [-r,0] \rightarrow \M_n(\R)$ which satisfies the following properties.
\begin{enumerate}
\item[(i)] $\forall t \in [0,T]$, $\eta(t,\cdot) \in NBV([-r,0], \M_n(\R))$
\item[(ii)] $\forall t \in [0,T]$, $\Vert \eta(t,\cdot) \Vert_{BV} = \Vert L(t) \Vert_{\mathfrak L}$
\item[(iii)] $[t \mapsto \eta(t,\cdot)] \in C^0([0,T],  NBV([-r,0], \M_n(\R)))$
\item[(iv)] $\forall t \in [0,T]$, $\forall \phi \in C^0([-r, T], \R^n)$, $L(t)\phi = \int_{-r}^0 d_2 \eta(t, \theta) \phi(\theta)$
\item[(v)] $\eta$ is Lebesgue measurable on $[0,T] \times [-r,0]$
\item[(vi)] $\eta$ is Riemann integrable on $[0,T] \times [-r,0]$.
\end{enumerate}
\end{proposition}
The following result is devoted to the resolvents of the equations of \eqref{eq32} and \eqref{eq33}. It is proven in \cite{BK1}.
\begin{theorem}\label{th32}
Under \eqref{eq31} there exists a mapping $X : [0,T] \times [0,T] \rightarrow \M_n(\R)$ which satisfies the following properties.
\begin{enumerate}
\item[(i)] $X$ is bounded on $[0,T] \times [0,T]$
\item[(ii)] $\forall s, t \in [0,T]$ such that $s \geq t$, $X(t,s) = I$ (identity)
\item[(iii)]  $\forall s \in [0,T]$, $X(\cdot,s) \in AC([0,T], \M_n(\R))$
\item[(iv)] $\forall s, t \in [0,T]$, $X(\cdot,s)$ is right-differentiable and left-differentiable at $t$, \\
$\frac{\partial X(\cdot,s)}{\partial t+} \in C^0_R([0,T], \M_n(\R))$, 
$\frac{\partial X(\cdot,s)}{\partial t-} \in C^0_L([0,T], \M_n(\R))$
\item[(v)] $\forall s, t \in [0,T]$ such that $s \geq t$, $\frac{\partial X(t,s)}{\partial t+} - \frac{\partial X(t,s)}{\partial t-} = \eta(t, (s-t)+) - \eta(t, (s-t))$
\item[(vi)] $\forall s \in [0,T]$, the set of the points of $[0,T]$ where $X(\cdot,s)$ is not differentiable is at most countable
\item[(vii)] $\forall t \in [0,T]$, $X(t,\cdot) \in BV([0,T], \M_n(\R)) \cap C^0_L([0,T], \M_n(\R))$.
\end{enumerate}
We define the mapping $Z : \{(t,\sigma) \in [0,T] \times [0,T] : t \geq \sigma \} \times C^0([-r,0] , \R^n) \rightarrow \R^n$ by setting
\[
Z(t,\sigma, \phi) := \int_{\sigma}^t X(t,\xi) \left( \int_{-r}^{\sigma - \xi} d_2 \eta(\xi, \theta) \phi(\xi - \sigma + \theta) \right) d \xi.
\]
Then the following properties hold for all $\sigma \in [0,T]$ and for all $\phi \in  C^0([-r,0] , \R^n)$.
\begin{enumerate}
\item[(viii)] $Z(\cdot,\sigma, \phi) \in C^0([ \sigma, T], \R^n)$
\item[(ix)] $\forall t \in [\sigma,T]$, $Z(\cdot,\sigma, \phi)$ is right-differentiable and left-differentiable at $t$, \\
$\frac{\partial Z(\cdot,\sigma, \phi)}{\partial t+} \in C^0_R(([ \sigma, T], \R^n)$, and $\frac{\partial Z(\cdot,\sigma, \phi)}{\partial t-} \in C^0_L(([ \sigma, T], \R^n)$
\item[(x)] $\forall t \in [\sigma,T]$, $\frac{\partial Z(t,\sigma, \phi)}{\partial t+} - \frac{\partial Z(t,\sigma, \phi)}{\partial t-} = 
\frac{\partial X(t,s)}{\partial t-} - \frac{\partial X(t,s)}{\partial t+}$.
\end{enumerate}
We define the set ${\mathfrak D} := \{ (t,\sigma) \in [0,T] \times [0,T] : t \geq \sigma -r \}$ and we define the mapping $U : {\mathfrak D} \times  C^0([-r,0] , \R^n) \rightarrow \R^n$ by setting
\[ 
U(t, \sigma, \phi) := \left\{
\begin{array}{lcl}
X(t, \sigma) \phi(0) + Z(t, \sigma, \phi) & {\rm if} & t \geq \sigma\\
\phi(t- \sigma) & {\rm if} & \sigma -r \leq t \leq \sigma.
\end{array}
\right.
\]
Then the following assertions hold for all $\sigma \in [0,T]$ for all $\phi \in C^0([-r,0] , \R^n)$.
\begin{enumerate}
\item[(xi)] $U(\cdot, \sigma, \phi)_{\vert_{[\sigma, t]}} \in C^1([\sigma, T], \R^n)$, and $U(\cdot, \sigma, \phi) \in C^0([\sigma -r,T], \R^n)$
\item[(xii)] $U(\cdot, \sigma, \phi)$ is the (unique) solution of \eqref{eq32} on $[\sigma, T]$.
\end{enumerate}
Moreover we define the mapping $V : {\mathfrak D} \times  C^0([-r,0] , \R^n) \times C^0([0,T], \R^n) \rightarrow \R^n$ by setting
\[ 
V(t, \sigma, \phi, h) := \left\{
\begin{array}{lcl}
U(t, \sigma, \phi) + \int_{\sigma}^t X(t, \alpha)h(\alpha) d \alpha & {\rm if} & t \geq \sigma\\
\phi(t- \sigma) & {\rm if} & \sigma -r \leq t \leq \sigma.
\end{array}
\right.
\]
Then the following assertions hold for all $\sigma \in [0,T]$ for all $\phi \in C^0([-r,0] , \R^n)$ and for all $h \in C^0([0,T), \R^n)$
\begin{enumerate}
\item[(xiii)] $V(\cdot, \sigma, \phi, h)_{\vert_{[\sigma, T]}} \in C^1([\sigma, T], \R^n)$ and $V(\cdot, \sigma, \phi, h) \in C^0([\sigma -r, T], \R^n)$
\item[(xiv)] $V(\cdot, \sigma, \phi, h)$ is the (unique) solution of \eqref{eq33} on $[\sigma, T]$.
\end{enumerate}
\end{theorem}
\vskip2mm
The proof of this theorem is contained into Section 6 of \cite{BK1}. The only differences between the results of \cite{BK1} and the present paper are the replacing of the condition $g(b) = 0$ by $g(a) = 0$ in the definition of $NBV([a,b], \R^n)$ and the extension of $\eta$ into $\eta^1$ with $\eta^1(t,\theta) = 0$ when $\theta < -r$ instead of $\eta(t,\theta) = 0$ when $\theta > 0$.
\subsection{The piecewise continuous time framework}
Instead of the continuity of the vector field $L$ we assume that
\begin{equation}\label{eq34}
L \in PC^0([0,T], {\mathfrak L}(C^0([-r,0],\R^n), \R^n)).
\end{equation}
When $\phi \in C^0([-r,0],\R^n)$, when $\sigma \in [0,T]$, and when $h \in PC^0([0,T], \R^n)$, we consider the following problems.
\begin{equation}\label{eq35}
\forall t \in [0, T] \setminus N_{x'}, x'(t) = L(t)x_t, \; \; x_{\sigma} = \phi
\end{equation}
\begin{equation}\label{eq36}
\forall t \in [0, T] \setminus N_{x'}, x'(t) = L(t)x_t + h(t), \; \; x_{\sigma} = \phi.
\end{equation}
A solution of \eqref{eq35} or of \eqref{eq36} is a function $x \in PC^1([0,T], \R^n)$; more precisely $x \in PC^1_{N_L}([0,T], \R^n)$ for \eqref{eq35} and 
$x \in PC^1_{N_L \cup N_h}([0,T], \R^n)$ for \eqref{eq36}. We can deduce the results of the piecewise continuous time from those of the continuous time framework by proceeding in the following way. If $0 = t_0 < t_1 < ...< t_p < t_{p+1} = T$ are the points of $N_L$, for all $k \in \{ 0,..., p \}$ we denote by $L_k$ the (continuous) restriction of $L$ at $[t_k, t_{k+1}]$. When we fix $\sigma \in [0,T)$ we consider the index $m$ such $\sigma < t_m < ... <t_{p+1} = T$ and we split the problem \eqref{eq35} into a finite list of problems like \eqref{eq32} as follows: first we have a solution $z^m$ of the problem $(x'(t) = L_{m-1}(t)x_t, x_{\sigma} = \phi)$ on $[\sigma, t_m]$, secondly we have the solution $z^{m+1}$ of the problem $(x'(t) = L_m (t)x_t, x_{\sigma} = \phi_m)$  on $[t_m, t_{m+1}]$, where $\phi_m(\theta) :=  z^m(t_m + \theta)$ for $\theta \in [-r,0]$, and inductively until to have a solution $z^{p+1}$ of the problem $(x'(t) = L_{p}(t)x_t, x_{\sigma} = \phi_p)$  on $[t_p, t_{p+1}]$, where $\phi_p(\theta) :=  z^p(t_p + \theta)$ for $\theta \in [-r,0]$. Then the function $z : [-t,T] \rightarrow \R^n$, defined by $z(t) := z^k(t)$ when $t \in [t_{k-1}, t_k]$, is a solution of \eqref{eq35}. 
\vskip1mm
Moreover using Proposition \ref{prop31}, for all $k \in \{ 0, ..., p \}$, we obtain the existence of $\eta_k :[t_k, t_{k+1}] \times [-r, 0] \rightarrow \M_n(\R)$ which satisfies the conclusions of Proposition \ref{prop31} where $[t_k, t_{k+1}]$ replaces $[0,T]$ and where $L_k$ replaces $L$. We define $\eta : [0,T] \times [-r,0] \rightarrow \M_n(\R)$ by setting $\eta(t, \theta) := \eta_k(t, \theta)$ when $t \in [t_k, t_{k+1})$, $ \theta \in [-r,0]$ when $k \in \{ 0, ..., p-1 \}$ and $\eta(t, \theta) := \eta_p(t, \theta)$ when $t \in [t_p, T]$, $\theta \in [-r,0]$. And so, from Proposition \ref{prop31} we deduce the following result.
\begin{proposition}\label{prop33}
Under \eqref{eq34} there exists a mapping $\eta : [0,T] \times [-r,0] \rightarrow \M_n(\R)$ which satisfies the following properties.
\begin{enumerate}
\item[(i)] $\forall t \in [0,T]$, $\eta(t,\cdot) \in NBV([-r,0], \M_n(\R))$
\item[(ii)] $\forall t \in [0,T]$, $\Vert \eta(t,\cdot) \Vert_{BV} = \Vert L(t) \Vert_{\mathfrak L}$
\item[(iii)] $[t \mapsto \eta(t,\cdot)] \in PC^0_{N_L}([0,T],  NBV([-r,0], \M_n(\R)))$
\item[(iv)] $\forall t \in [0,T]$, $\forall \phi \in C^0([-r, T], \R^n)$, $L(t)\phi = \int_{-r}^0 d_2 \eta(t, \theta) \phi(\theta)$
\item[(v)] $\eta$ is Lebesgue measurable on $[0,T] \times [-r,0]$
\item[(vi)] $\eta$ is Riemann integrable on $[0,T] \times [-r,0]$.
\end{enumerate}
\end{proposition}
Now we want to obtain a result which is analogous to Theorem \ref{th32} for the piecewise continuous time framework. We proceed as in \cite{BK1}, replacing the property $[t \mapsto \eta(t,\cdot)] \in C^0([0,T], NBV([-r,0], \M_n(\R))$ by the property $[t \mapsto \eta(t,\cdot)] \in PC^0_{N_L}([0,T], NBV([-r,0], \M_n(\R))$ and then we obtain the following result.
\begin{theorem}\label{th34}
Under \eqref{eq34} there exists a mapping $X : [0,T] \times [0,T] \rightarrow \M_n(\R)$ which satisfies the following properties.
\begin{enumerate}
\item[(i)] $X$ is bounded on $[0,T] \times [0,T]$
\item[(ii)] $\forall s, t \in [0,T]$ such that $s \geq t$, $X(t,s) = I$ (identity)
\item[(iii)]  $\forall s \in [0,T]$, $X(\cdot,s) \in AC([0,T], \M_n(\R))$
\item[(iv)] $\forall s, t \in [0,T]$, $X(\cdot,s)$ is right-differentiable and left-differentiable at $t$, \\
$\frac{\partial X(\cdot,s)}{\partial t+} \in C^0_R([0,T], \M_n(\R))$, 
$\frac{\partial X(\cdot,s)}{\partial t-} \in C^0_L([0,T], \M_n(\R))$
\item[(v)] $\forall s \in [0,T]$, $\forall t \in [0,T] \setminus N_L$ such that $s \geq t$, \\
$\frac{\partial X(t,s)}{\partial t+} - \frac{\partial X(t,s)}{\partial t-} = \eta(t, (s-t)+) - \eta(t, (s-t))$
\item[(vi)] $\forall s \in [0,T]$, the set of the points of $[0,T]$ where $X(\cdot,s)$ is not differentiable is at most countable
\item[(vii)] $\forall t \in [0,T]$, $X(t,\cdot) \in BV([0,T], \M_n(\R)) \cap C^0_L([0,T], \M_n(\R))$.
\end{enumerate}
We define the mapping $Z : \{(t,\sigma) \in [0,T] \times [0,T] : t \geq \sigma \} \times C^0([-r,0] , \R^n) \rightarrow \R^n$ by setting
\[
Z(t,\sigma, \phi) := \int_{\sigma}^t X(t,\xi) \left( \int_{-r}^{\sigma - \xi} d_2 \eta(\xi, \theta) \phi(\xi - \sigma + \theta) \right) d \xi.
\]
Then the following properties hold for all $\sigma \in [0,T]$ and for all $\phi \in  C^0([-r,0] , \R^n)$.
\begin{enumerate}
\item[(viii)] $Z(\cdot,\sigma, \phi) \in C^0([ \sigma, T], \R^n)$
\item[(ix)] $\forall t \in [\sigma,T]$, $Z(\cdot,\sigma, \phi)$ is right-differentiable and left-differentiable at $t$, \\
$\frac{\partial Z(\cdot,\sigma, \phi)}{\partial t+} \in C^0_R(([ \sigma, T], \R^n)$, and $\frac{\partial Z(\cdot,\sigma, \phi)}{\partial t-} \in C^0_L(([ \sigma, T], \R^n)$
\item[(x)] $\forall t \in [\sigma,T]\setminus N_L$, $\frac{\partial Z(t,\sigma, \phi)}{\partial t+} - \frac{\partial Z(t,\sigma, \phi)}{\partial t-} = 
\frac{\partial X(t,s)}{\partial t-} - \frac{\partial X(t,s)}{\partial t+}$.
\end{enumerate}
We define the set ${\mathfrak D} := \{ (t,\sigma) \in [0,T] \times [0,T] : t \geq \sigma -r \}$ and we define the mapping $U : {\mathfrak D} \times  C^0([-r,0] , \R^n) \rightarrow \R^n$ by setting
\[ 
U(t, \sigma, \phi) := \left\{
\begin{array}{lcl}
X(t, \sigma) \phi(0) + Z(t, \sigma, \phi) & {\rm if} & t \geq \sigma\\
\phi(t- \sigma) & {\rm if} & \sigma -r \leq t \leq \sigma.
\end{array}
\right.
\]
Then the following assertions hold for all $\sigma \in [0,T]$ for all $\phi \in C^0([-r,0] , \R^n)$.
\begin{enumerate}
\item[(xi)] $U(\cdot, \sigma, \phi)_{\vert_{[\sigma, t]}} \in PC^1([\sigma, T], \R^n)$, and $U(\cdot, \sigma, \phi) \in C^0([\sigma -r,T], \R^n)$
\item[(xii)] $U(\cdot, \sigma, \phi)$ is the (unique) solution of \eqref{eq35} on $[\sigma, T]$.
\end{enumerate}
Moreover we define the mapping $V : {\mathfrak D} \times  C^0([-r,0] , \R^n) \times C^0([0,T], \R^n) \rightarrow \R^n$ by setting
\[ 
V(t, \sigma, \phi, h) := \left\{
\begin{array}{lcl}
U(t, \sigma, \phi) + \int_{\sigma}^t X(t, \alpha) h(\alpha) d \alpha & {\rm if} & t \geq \sigma\\
\phi(t- \sigma) & {\rm if} & \sigma -r \leq t \leq \sigma.
\end{array}
\right.
\]
Then the following assertions hold for all $\sigma \in [0,T]$ for all $\phi \in C^0([-r,0] , \R^n)$ and for all $h \in C^0([0,T), \R^n)$
\begin{enumerate}
\item[(xiii)] $V(\cdot, \sigma, \phi, h)_{\vert_{[\sigma, T]}} \in PC^1([\sigma, T], \R^n)$ and $V(\cdot, \sigma, \phi, h) \in C^0([\sigma -r, T], \R^n)$
\item[(xiv)] $V(\cdot, \sigma, \phi, h)$ is the (unique) solution of \eqref{eq36} on $[\sigma, T]$.
\end{enumerate}
\end{theorem}
\subsection{Adjoint equation.}
\begin{proposition}\label{prop35}
Let $X$ be provided by Theorem \ref{th34} and $\eta$ be provided by Proposition \ref{prop31}. We define $Y : [0,T] \times [0,T] \rightarrow \M_n(\R)$ by setting $Y(s,t) := X(t,s)$. Then when $s \leq t$, the following equation is satisfied: $Y(s,t) = I - \int_s^t Y(\alpha, t) \eta(\alpha, s - \alpha) d\alpha$.
\end{proposition}
\begin{proof}
We set $k(\alpha,s) := \eta(\alpha, s- \alpha)$. After Theorem 5.4 in \cite{BK1}, there exists $R$ which satisfies $R(t,s) = k(t,s) - \int_s^t R(t,\alpha) k(\alpha,s) d \alpha$. Ever after \cite{BK1} we have $X(t,s) := I - \int_s^t R(\alpha,s) d \alpha$. Then we calculate
\[
\begin{array}{cl}
\null & I - \int_s^t Y(\alpha,t) k (\alpha,s) d \alpha = I - \int_s^t X(t,\alpha) k (\alpha,s) d \alpha \\
= &  I - \int_s^t ( I - \int_{\alpha}^t R(\beta,\alpha) d \beta ) k (\alpha,s) d \alpha \\
= & I - \int_s^t k(\alpha,s) d \alpha +  \int_s^t (\int_{\alpha}^t R(\beta,\alpha) k (\alpha,s) d \beta ) d \alpha \\
= &  I - \int_s^t k(\beta,s) d \beta + \int_s^t ( \int_s^{\beta} R(\beta,\alpha) k (\alpha,s) d \alpha ) d \beta\\
= & I - \int_s^t [ k(\beta,s) -  \int_s^t  \int_s^{\beta} R(\beta,\alpha) k (\alpha,s) d \alpha] d \beta\\
= & I - \int_s^t R(\beta,s) d \beta = X(t,s) = Y(s,t).
\end{array}
\]
\end{proof}
The integral equation which is present into the previous statement is called the adjoint equation of \eqref{eq35}.
\section{Pontryagin principle}
First we give the qualification condition of Michel \cite{Mi}, where $(\overline{x}, \overline{u})$ is an admissible process of $({\mathfrak M})$.
\[
{\rm (QC)}
\left\{
\begin{array}{rl}
\null & \forall (c_j)_{0 \leq j \leq ni + n_e} \in \R^{1 + p+ q}\\
{\rm if} & \forall j = 0,...,n_i, \; c_j \geq 0\\
\null & \forall j= 1,...,n_i, \; c_j g^j(\overline{x}(T))= 0\\
\null & \sum_{j=0}^{n_i + n_e} c_j Dg^j(\overline{x}(T)) = 0\\
{\rm then} & \forall j = 0,...,n_i + n_e, \; c_j = 0.
\end{array}
\right.
\]
The main result of the paper is the following statement of a Pontryagin principle.
\begin{theorem}\label{th41}
Under the assumptions \eqref{eq11} and \eqref{eq12}, if $(\overline{x}, \overline{u})$ is a solution of the Mayer problem $({\mathfrak M})$ then there exist $\lambda_0$,..., $\lambda_{n_i + n_e} \in \R$ and there exists $p \in BV([0, T], \R^{n*}) \cap C^0_L([0,T], \R^{n*})$ which satisfy the following conditions.
\begin{enumerate}
\item[(NN)] $(\lambda_0,..., \lambda_{n_i + n_e}) \neq (0,...,0)$
\item[(Si)] $\forall j \in \{0,...,n_i \}$, $\lambda_j \geq 0$
\item[(Sl)] $\forall j \in \{0,...,n_i \}$, $\lambda_j g^j(\overline{x}(T)) = 0$
\item[(AE)] $\forall t \in [0,T]$, $\frac{d}{dt} (p(t) + \int_t^{\min\{t+r, T \}} p(\alpha) \eta(\alpha, t - \alpha) d \alpha) = 0$
\item[(T)] $p(T) = \sum_{j=0}^{n_i + n_e} \lambda_j Dg^j(\overline{x}(T))$
\item[(MP)] $\forall t \in [0,T]$, $\forall u \in U$, \hskip2mm $p(t) f(t, \overline{x}_t, \overline{u}(t)) \geq p(t)f(t, \overline{x}_t, u)$
\end{enumerate}
If moreover (QC) is fulfilled, we obtain the following additional conclusions.
\begin{enumerate}
\item[(A2)] There exists $\epsilon > 0$ such that $p(t) \neq 0$ for all $t \in (T- \epsilon, T]$.
\item[(A2)] $p(T) \neq 0$ and, for all $t \in [0,T]$, $p_{\mid_{[t, \min \{t +r, T \}]}} \neq 0$.
\end{enumerate}
\end{theorem}
\vskip3mm
In this statement, (NN) is a condition of non nullity of the multipliers, (Si) is a condition on the signs of multipliers, (Sl) is a condition of slackness on the final inequality constraints, (AE) is called the adjoint equation, (T) is the transversality condition, and (MP) is the maximum principle. The mapping $\eta$ which is present into (AE) comes from Proposition \ref{prop33} with $L(t) = D_2f(t, \overline{x}_t, \overline{u}(t))$.
\begin{remark}\label{rem42}
The conclusion (AE) is equivalent to say that the function $[t \mapsto p(t) + \int_t^{\min\{ t+r, T \}} p(\xi) \eta(\xi, \xi -t) d \xi]$ is constant on $[0,T]$. Since $\int_T^{\min\{ T+r, T \}} p(\xi) \eta(\xi, \xi -t) d \xi = 0$, using (TC), we obtain that (AE) is equivalent to
$$\forall t \in [0,T],\; \; p(t) + \int_t^{\min\{ t+r, T \}} p(\xi) \eta(\xi, \xi -t) d \xi = p(T) = \sum_{j=0}^{n_i + n_e} \lambda_j Dg^j(\overline{x}(T)).$$
\end{remark}
\section{proof of the Pontryagin principle}
Our proof follows the proof given by Michel in \cite{Mi}; we provide the useful changes to adapt it to the setting of systems governed by a functional differential equation. Note that we can choose $\overline{u}$ as a right-continuous function without to lost generality. We arbitrarily fix $S = \{ (t_i,v_i) : i \in \{ 1,...,N \} \}$ where $0 \leq t_1 \leq ...\leq t_N < T$ and the $v_i \in U$; we denote by ${\mathfrak S}$ the set of all these $S$. When $a = (a_1,...,a_N) \in \R^N_+$, we define $J(i) := \{ j \in \{1,...,i \} : t_j = t_i \}$ and we set $b_i := 0$ when $J(i) = \emptyset$ and $b_i := \sum_{j \in J(i)}a_j$ when $J(i) \neq \emptyset$. After that we define the subintervals $I_i := [t_i + b_i, t_i + b_i + a_i)$ and $u(t,S,a) := v_i$ when $t \in I_i$ and the control $u(t,S,a) := \overline{u}(t)$ when $t \in [0,T] \setminus \cup_{1 \leq i \leq N} I_i$. Taking $a$ small enough we have $I_i \subset [0,T]$ and the $I_i$ are pairwise disjoint. We denote by $x(\cdot,S,a)$ the unique solution on $[0,T]$ of the following problem of Cauchy
$$\forall t \in [0,T] \setminus N_{u(\cdot,S,a)}, \frac{\partial x(t,S,a)}{\partial t} = f(t, x(\cdot,S,a)_t, u(t,S,a)), \: \; x(\cdot,S,a)_0 = \phi.$$
\begin{lemma}\label{lem51}
We fix $S \in {\mathfrak S}$ and $a \in \R^N_+$. We denote by $z(\cdot,S,a)$ the solution on $[0,T]$ of the following linear problem
\[
\left\{
\begin{array}{l}
\forall t \in [0,T] \setminus N_{u(.,S,a)},\\
 \frac{\partial  z(t,S,a)}{\partial t} = D_2f(t, \overline{x}_t, \overline{u}(t))z(\cdot,S,a)_t + [f(t, \overline{x}_t, u(t,S,a)) -f(t, \overline{x}_t,\overline{u}(t))]\\
 z(\cdot,S,a)_0 = 0.
\end{array}
\right.
\]
Then the partial differential $D_a z(T,S,0)$ exists and the following equality holds
$$D_a z(T,S,0)a = \sum_{i=1}^N a_i X(T,t_i)[f(t_i, \overline{x}_{t_i}, v_i) - f(t_i, \overline{x}_{t_i},\overline{u}(t_i))].$$
where $X(T,t_i)$ comes from Theorem \ref{th34} when $L(t) =  D_2f(t, \overline{x}_t, \overline{u}(t))$.
\end{lemma}
\begin{proof}
We set $\Delta(t,S,a) := [f(t, \overline{x}_t, u(t,S,a)) -f(t, \overline{x}_t,\overline{u}(t))]$ and using Theorem \ref{th34} we obtain $z(T,S,a) = V(T,0,0,\Delta(\cdot,S,a))$= $X(T,0)0 + \int_0^T X(T,\xi) (\int_{-r}^0 d_2\eta(\xi,\theta) 0(\xi + \theta) d \xi + \int_0^T X(T,\xi) \Delta(\xi, S,a) d \xi$ = $\int_0^T X(T,\xi) \Delta(\xi, S,a) d \xi$. Since $\Delta(\xi,S,a) = 0$ when $\xi \in [0,T] \setminus \cup_{1 \leq i \leq N} I_i$, we obtain 
$$z(T,S,a) = \sum_{i=1}^N \int_{t_i + b_i}^{t_i + b_i + a_i} X(T,\xi)  [f(t_i, \overline{x}_{t_i}, v_i) - f(t_i, \overline{x}_{t_i},\overline{u}(t_i))] d\xi.$$
Note that  
\[
\begin{array}{cl}
\null & 
a_i  X(T,t_i) [f(t_i, \overline{x}_{t_i}, v_i) - f(t_i, \overline{x}_{t_i},\overline{u}(t_i))]\\
 =& \int_{t_i + b_i}^{t_i + b_i + a_i}X(T,t_i)[f(t_i, \overline{x}_{t_i}, v_i) - f(t_i, \overline{x}_{t_i},\overline{u}(t_i))]d \xi.
\end{array}
\]
Since $u(t,S,0) = \overline{u}(t)$ we have $\Delta(t,S,0) = 0$ and then $z(T,S,0) = 0$. For all $i \in \{1,...,N \}$ we define $\rho_i := {\mathfrak S} \times \R^N_+ \rightarrow \R^n$ by setting 
$$\varrho_i(S,a) := \frac{1}{a_i} \int_{t_i + b_i}^{t_i + b_i + a_i} [X(t,s) \Delta(s,S,a) - X(T,t_i)( f(t_i, \overline{x}_{t_i}, v_i) - f(t_i, \overline{x}_{t_i},\overline{u}(t_i)))] ds$$
when $a_i \neq 0$ and $\varrho_i(S,a) := 0$ when $a_i = 0$.
Then the following formula holds.
\begin{equation}\label{eq51}
z(T,S,a) = z(T,S,0) + \sum_{i=1}^N a_i X(T,t_i) \Delta(t_i,S,a) + \sum_{i=1}^N a_i \varrho_i(S,a).
\end{equation}
Now to prove the result, it sufffices to prove that the following assertion holds.
\begin{equation}\label{eq52}
\forall i \in \{1,...,N \}, \; \; \lim_{a \rightarrow 0} \varrho_i(S,a) = 0.
\end{equation}
In the formula of $\varrho_i$ we do the change of variable $s =t_i + b_i + \theta a_i$ with $\theta \in [0,1]$, setting $\varpi_i(S,a,\theta) := X(T,t_i + b_i+ \theta a_i)\Delta(t_i + b_i+ \theta a_i, S,a) - X(T,t_i)\Delta(t_i,S,a)$, we obtain the following formula.
\begin{equation}\label{eq53}
\varrho_i(S,a) = \int_0^1 \varpi_i(S,a,\theta) d \theta.
\end{equation}
we arbitrarily fix $\theta \in [0,1)$. Note that, for all $i \in \{ 1,..., N \}$, we have $\lim_{a \rightarrow 0}a_i = 0$ and $\lim_{a \rightarrow 0}b_i = 0$. Since $X(T,\cdot)$ is right-continuous ((vi) of Theorem \ref{th34}) we obtain $\lim_{a \rightarrow 0}X(T, t_i + b_i + \theta a_i) = X(T,t_i)$. Note that $\Delta( t_i + b_i + \theta a_i, S, a) = f( t_i + b_i + \theta a_i, \overline{x}_{ t_i + b_i + \theta a_i}, v_i) - f( t_i + b_i + \theta a_i, \overline{x}_{ t_i + b_i + \theta a_i}, \overline{u}( t_i + b_i + \theta a_i))$. Since $f$ is continuous, since $[t \mapsto \overline{x}_t]$ is continuous and since $\overline{u}$ is right-continuous, we obtain $\lim_{a \rightarrow 0}\Delta( t_i + b_i + \theta a_i, S, a)
= f(t_i, \overline{x}_{t_i}, v_i) - f(t_i, \overline{x}_{t_i}, \overline{u}(t_i))$ which implies $\lim_{a \rightarrow 0}(X(T,t_i + b_i+ \theta a_i)\Delta(t_i + b_i+ \theta a_i, S,a) - X(T,t_i)\Delta(t_i,S,a)) = X(T,t_i)[ f(t_i, \overline{x}_{t_i}, v_i) - f(t_i, \overline{x}_{t_i}, \overline{u}(t_i))] - X(T,t_i)[ f(t_i, \overline{x}_{t_i}, v_i) - f(t_i, \overline{x}_{t_i}, \overline{u}(t_i))] = 0$. And so we obtain the following property.
\begin{equation}\label{eq54}
\forall \theta \in [0,1), \; \; \lim_{a \rightarrow 0} \varpi_i(S,a,\theta)  = 0.
\end{equation}
We fix $\delta \in (0, + \infty)$ small enough to have $\{ t_i + b_i+ \theta a_i : \theta \in [0,1], a \in \R^N_+ \cap \overline{B}(0, \delta) , i \in \{ 1,...,N \} \} \subset [0,T]$, and then the closure of this subset is compact. Since $[t \mapsto \overline{x}_t ]$ is continuous, $\{ \overline{x}_t : t \in [0,T] \}$ is compact. Since $\overline{u}$ is piecewise continuous on $[0,T]$, $\overline{u}([0,T])$ is compact, and $\{u(t_i + b_i+ \theta a_i,S,a) : i \in \{1,...,N \}, a \in \R^N_+ \cap \overline{B}(0, \delta), \theta \in [0,1] \}$ is compact. Since $f$ is continuous we obtain that there exists $c \in (0, + \infty)$ such that
\[
\begin{array}{l}
\forall a \in \R^N_+ \cap \overline{B}(0, \delta), \forall \theta \in [0,1], \forall i \in \{ 1,...,N \}, \\
\Vert f(t_i + b_i+ \theta a_i, \overline{x}_{t_i + b_i+ \theta a_i}, u(t_i + b_i+ \theta a_i,S,a))\\
 - f(t_i + b_i+ \theta a_i, \overline{x}_{t_i + b_i+ \theta a_i}, \overline{u}(t_i + b_i+ \theta a_i)) \Vert \leq c.
\end{array}
\]
Since $X$ is bounded on $[0,T] \times [0,T]$, we obtain the existence of $c_1 \in (0, + \infty)$ such that
\[
\forall a \in \R^N_+ \cap \overline{B}(0, \delta), \forall \theta \in [0,1], \forall i \in \{ 1,...,N \}, 
\Vert \varpi_i(S,a, \theta) \Vert \leq c_1.
\]
Since a constant is Lebesgue integrable on $[0,1]$,  we can use the theorem of the dominated convergence of Lebesgue to assert that $\lim_{a \rightarrow 0}\varrho_i(S,a) = \int_0^1 0 d \theta = 0$. Then \eqref{eq52} is proven, and the conclusion of the lemma follows from \eqref{eq51}. 
\end{proof}
\begin{lemma}\label{lem52}
Let $S \in {\mathfrak S}$. There exists $\delta_1 \in (0, + \infty)$ and $c_2 \in (0, + \infty)$ such that, for all $a \in \R^N_+ \cap \overline{B}(0, \delta_1)$,
$\int_0^T \Vert f(t, \overline{x}_t, u(t,S,a)) - f(t, \overline{x}_t, \overline{u}(t)) \Vert dt \leq c_2 \cdot \Vert a \Vert$.
\end{lemma}
\begin{proof}
Note that the integrand in the formula is $\Vert \Delta (t,S,a) \Vert$. We introduce $e : [0,T] \times  (\R^N_+ \cap \overline{B}(0, \delta)) \rightarrow \R_+$ by setting $e(t,a):= \int_0^t \Vert \Delta (t,S,a) \Vert dt $.\\
 Since $\Delta(T,S,0) = 0$ we have $\sigma (T,0) = 0$. We set $\Xi_i := f(t_i, \overline{x}_{t_i}, v_i) - f(t_i, \overline{x}_{t_i}, \overline{u}(t_i))$. Then we have
\[
\begin{array}{cl}
\null & e(T,a) - e(T,0) - \sum_{i=1}^N a_i \Vert \Xi_i \Vert\\
=& \int_0^T \Vert \Delta (t,S,a) \Vert dt  - 0- \sum_{i=1}^N a_i \Vert \Xi_i \Vert\\
= & \int_{[0,T] \setminus \cap_{i=1}^N I_i} \Vert \Delta (t,S,a) \Vert dt  + \int_{\cap_{i=1}^N I_i} \Vert \Delta (t,S,a) \Vert dt - \sum_{i=1}^N a_i \Vert \Xi_i \Vert\\
=& 0+ \sum_{i=1}^N \int_{I_i} \Vert \Delta (t,S,a) \Vert dt  - \sum_{i=1}^N a_i \Vert \Xi_i \Vert\\
=&  \sum_{i=1}^N \int_{t_i+ b_i}^{t_i + b_+ + a_i}  \Vert \Delta (t,S,a) \Vert dt  - \sum_{i=1}^N  \int_{t_i+ b_i}^{t_i + b_+ + a_i} \Vert \Xi_i \Vert dt\\
=& \sum_{i=1}^N a_i \frac{1}{a_i} \int_{t_i+ b_i}^{t_i + b_+ + a_i}  (\Vert \Delta (t,S,a) \Vert -\Vert \Xi_i \Vert) dt.
\end{array}
\]
And so defining the function $\nu_i : {\mathfrak S} \times (\R^N_+ \cap \overline{B}(0, \delta)) \rightarrow \R_+$ by setting 
\[
\nu_i(S,a) :=
\left\{
\begin{array}{lcl}
\frac{1}{a_i} \int_{t_i+ b_i}^{t_i + b_+ + a_i}  (\Vert \Delta (t,S,a) \Vert -\Vert \Xi_i \Vert) dt& {\rm if} & a_i \neq 0\\
0 & {\rm if} & a_i = 0,
\end{array}
\right.
\]
we obtain, for all $a \in \R^N_+ \cap \overline{B}(0, \delta)$,
\begin{equation}\label{eq55}
e(T,a) = e(T,0) + \sum_{i=1}^N a_i \Vert \Xi_i \Vert + \sum_{i=1}^N a_i \nu_i(S,a).
\end{equation}
Using the change of variable $s = t_i + b_i + \theta a_i$ we obtain $\nu_i(S,a) = \int_0^1 (\Vert \Delta (t_i + b_i + \theta a_i,S,a) \Vert  -\Vert \Xi_i \Vert) d\theta$. In the previous proof we have yet seen that $\lim_{a \rightarrow 0} (\Vert \Delta (t_i + b_i + \theta a_i,S,a) \Vert  -\Vert \Xi_i \Vert) = 0$, and that there exists $c \in (0, + \infty)$ such that, for all $a \in \R^N_+ \cap \overline{B}(0, \delta)$, for all $\theta \in [0,1]$, $\vert \Vert \Delta (t_i + b_i + \theta a_i,S,a) \Vert  -\Vert \Xi_i \Vert  \vert \leq c$. And so we can use the theorem of the dominated convergence of Lebesgue and assert that $\lim_{a \rightarrow 0} \nu_i(S,a) = 0$. And then, using \eqref{eq55} we can say that $e(T,\cdot)$ is differentiable at 0 and that its partial differential is $D_2 e(T,0).a = \sum_{i=1}^N a_i \Vert \Xi_i \Vert$ that implies the existence of a mapping $\mu :  {\mathfrak S} \times (\R^N_+ \cap \overline{B}(0, \delta)) \rightarrow \R$ such that $\lim_{a \rightarrow 0} \mu(a) = 0$ and $e(T,a) - e(T,0) - \sum_{i=1}^N a_i \Vert \Xi_i \Vert = \Vert a \Vert \cdot \mu(a)$. We fix $\delta_1 \in (0 , \delta]$ such that $\Vert a \Vert \leq \delta_1 \Longrightarrow \vert \mu(a) \vert \leq 1$. Then, when $a \in \R^n_+ \cap \overline{B}(0, \delta_1)$ we have $e(T,a) = \vert \sigma(T,a) \vert \leq \Vert D_2 \sigma(T,0) \Vert \cdot \Vert a \Vert + \Vert a \vert$. To conclude it suffices to take $c_2 := \Vert D_2 e(T,0) \Vert + 1$.
\end{proof} 
\begin{lemma}\label{lem53}
Let $\Gamma : [0,T] \times C^0([-r,0], \R^n) \rightarrow \R^n$. Let $0 = \tau_0 < \tau_1 <...<\tau_{\ell} < \tau_{\ell +1} = T$ and $F := \{ \tau_i : i \in \{0,...,{\ell} \} \}$.  We assume that, for all $i \in \{0, ..., {\ell} \}$, $\Gamma$ is continuous on $[\tau_i, \tau_{i+1}] \times C^à([-r,0], \R^n)$.Then the Nemytskii operator ${\mathcal N}_{\Gamma} : C^0([0,T], C^0([-r,0], \R^n)) \rightarrow PC^0_F([0,T],\R^n)$, defined by ${\mathcal N}_{\Gamma}(\Phi) := [ t \mapsto \Gamma(t, \Phi(t))]$, is continuous. 
\end{lemma}
\begin{proof}
Let $\Phi \in C^0([0,T], C^0([-r,0], \R^n))$. Then, for all $i \in \{ 0,...,{\ell} \}$, the restriction of  $\Phi$ at $[\tau_i, \tau_{i+1}]$ belongs to $C^0(
[\tau_i, \tau_{i+1}],  C^0([-r,0], \R^n))$. Since $\Gamma$ is continuous on $[\tau_i, \tau_{i+1}] \times C^0([-r,0], \R^n)$, we know that ${\mathcal N}_{\Gamma}(\Phi) \in C^0([\tau_i, \tau_{i+1}], \R^n)$. Moreover we obtain that ${\mathcal N}_{\Gamma}(\Phi) \in PC^0_F([0,T], \R^n)$.\\
 We arbitrarily fix $\Phi \in C^0([0,T], C^0([-r,0], \R^n))$. Using Lemma 3.10 in \cite{BCNP}, we know that: $\forall i \in \{ 0,...,{\ell} \}$,$ \forall \epsilon > 0$, $\exists \beta_{\epsilon, i} > 0$, such that, $\forall \Psi_i \in C^0([\tau_i, \tau_{i+1}], C^0([-r,0], \R^n))$, $\sup_{t \in [\tau_i, \tau_{i+1}]}\Vert \Psi_i(t) - P(t) \Vert \leq \beta_{\epsilon, i} \Longrightarrow \sup_{t \in [\tau_i, \tau_{i+1}]} \Vert \Gamma(t, \Psi_i(t)) - \Gamma(t, \Phi(t)) \Vert \leq \epsilon$.
We arbitrarily fix $\epsilon > 0$ and we set $\beta_{\epsilon} := \min_{0 \leq i \leq {\ell}} \beta_{\epsilon, i} > 0$. Let $\Psi \in C^0([0,T], C^0([-r,0], \R^n))$ such that $\sup_{t \in [0,T]} \Vert \Psi(t) - \Phi(t) \Vert \leq \beta_{\epsilon}$; then we have, $\forall i \in \{ 0,...,{\ell} \}$, $\sup_{t \in [\tau_i, \tau_{i+1}]}\Vert \Psi(t) - \Phi(t) \Vert \leq \beta_{\epsilon, i}$ that implies: $\forall i \in \{ 0,...,{\ell} \}$, $\sup_{t \in [\tau_i, \tau_{i+1}]} \Vert \Gamma(t, \Psi(t)) - \Gamma(t, \Phi(t)) \Vert \leq \epsilon$, and consequently we obtain that\\
$\sup_{t \in [0,T]} \Vert \Gamma(t, \Psi(t)) - \Gamma(t, \Phi(t)) \Vert \leq \epsilon$.
\end{proof}
\begin{lemma}\label{lem54}
There exist $\delta_2 \in (0, + \infty)$ and $c_4 \in (0, + \infty)$ such the following assertions hold.
\begin{enumerate}
\item[(i)] $\forall a \in \R^N_+ \cap \overline{B}(0, \delta_2)$, $x(\cdot,S,a)$ is defined on $[0,T]$ all over.
\item[(ii)]  $\forall a \in \R^N_+ \cap \overline{B}(0, \delta_2)$, $\forall t \in [0,T]$, $\Vert x(t,S,a)- \overline{x}(t) \Vert \leq c_4. \Vert a \Vert$.
\end{enumerate}
\end{lemma}
\begin{proof} Let $a \in \R^N_+ \cap \overline{B}(0, \delta_1)$ where $\delta_1$ is provided by the previous lemma. By induction we built the sequence $(x^m(\cdot,S,a))_{m \in \N}$ of functions from $[0,T]$ into $\R^n$ in the following way.
\[
\left\{
\begin{array}{l}
x^0(\cdot,S,a) := \overline{x}\\
\forall m \geq 1, \forall t \in [0,T], x^m(t,S;a) = \overline{x}(0) + \int_0^t f(s, x^{m-1}(\cdot,S,a)_s, u(s,S,a)) ds\\
\forall m \geq 1, x^m_0 = \phi
\end{array}
\right.
\]
We have $\Vert x^1(t,S,a) - x^0(t,S,a) \Vert = \Vert  x^1(t,S,a) -  \overline{x}(t) \Vert = \Vert \int_0^t f(s, \overline{x}_s, u(s,S,a)) ds - \int_0^t  f(s, \overline{x}_s,\overline{u}(s))ds \Vert \leq \int_0^t \Vert  f(s, \overline{x}_s, u(s,S,a)) -   f(s, \overline{x}_s,\overline{u}(s)) \Vert ds \leq $ \\
$   \int_0^T \Vert  f(s, \overline{x}_s, u(s,S,a)) -   f(s, \overline{x}_s,\overline{u}(s)) \Vert ds \leq c_2 \cdot \Vert a  \Vert$. And so we have proven the followowing assertion: $
\forall a \in \R^N_+ \cap \overline{B}(0, \delta_1), \forall t \in [0,T], \Vert x^1(t,S,a) - x^0(t,S,a) \Vert \leq c_2. \Vert a  \Vert$. Then, for all $t \in [0,T]$, $\forall \theta \in [-r,0]$, if $ t+ \theta \geq 0$ then $\Vert x^1(t+ \theta,S,a) - x^0(t + \theta,S,a) \Vert \leq c_2 \cdot \Vert a  \Vert$, and if $t+ \theta < 0$ then 
$\Vert x^1(t+ \theta,S,a) - x^0(t + \theta,S,a) \Vert = \Vert \phi(\theta) - \phi(\theta) \Vert = 0 \leq c_2 \cdot \Vert a  \Vert$. And so we have proven the following assertion.
\begin{equation}\label{eq56}
\forall a \in \R^N_+ \cap \overline{B}(0, \delta_1), \forall t \in [0,T],  \Vert x^1(\cdot,S,a)_t - x^0(\cdot,S,a)_t \Vert \leq c_2 \cdot \Vert a  \Vert.
\end{equation}
We introduce the set
$${\mathfrak B} := \{ (t, \phi, v) \in [0,T] \times C^0([-r,0], \R^n) \times U : \Vert \phi - \overline{x}_t \Vert_{\infty} \leq r, v \in \{\overline{u}(t), v_1,...,v_N \} \}.$$
Note that ${\mathfrak B}$ is bounded since it is included into the bounded set $[0,T] \times \overline{B}(\overline{x}, r) \times (\overline{u}([0,T]) \cup \{ v_1...,v_N \})$. Since $D_2f$ is a bounded operator, $c_3 := \sup \{ \Vert D_2f(t,\phi,v) \Vert : (t, \phi, v) \in {\mathfrak B} \} < + \infty$.  Now we want to prove by induction the following relation, for all $a \in \R^N_+ \cap \overline{B}(0, \delta_1)$ such that $\Vert a \Vert \leq \frac{r}{c_2} e^{-c_3T}$, $\forall m \in \N$, $\forall t \in [0,T]$:
\begin{equation}\label{eq57}
\left.
\begin{array}{r}
\Vert x^m(\cdot,S,a)_t - x^0(\cdot,S,a)_t \Vert \leq r, \\
\Vert x^{m+1}(\cdot,S,a)_t - x^m(\cdot,S,a)_t \Vert \leq \frac{1}{m !}c_2 \Vert a \Vert (c_3 t)^m.
\end{array}
\right\}
\end{equation}
For $m=0$, using Lemma \ref{lem52} and \eqref{eq56} we obtain, for all $t \in[0,T]$, $\Vert x^1(\cdot,S,a)_t- x^0(\cdot,S,a)_t \Vert \leq c_2 \cdot \Vert a \Vert = \frac{1}{0!} c_2 \cdot \Vert a \Vert (c_3 t)^0$, and $c_2 \cdot \Vert a \Vert \leq r e^{-c_3 T} \leq r$.\\  
Now we assume that \eqref{eq57} holds for the integer $n \leq p-1$. Then, for all $t \in [0,T]$, we have
\[
\begin{array}{cl}
\null & \Vert x^p(t,S,a) - x^0(t,S,a) \Vert = \Vert \sum_{m=0}^{p-1} (x^{m+1}(t,S,a) - x^m(t,S,a)) \Vert \\
\leq & \sum_{m=0}^{p-1} \Vert x^{m+1}(t,S,a) - x^m(t,S,a) \Vert \leq \sum_{m=0}^{p-1}(\frac{1}{m !}c_2 \Vert a \Vert (c_3 t)^m)\\
 = & c_2 \cdot \Vert a \Vert \sum_{m=0}^{p-1} \frac{(c_3 t)^m}{m !} \leq  c_2 \cdot \Vert a  \Vert e^{c_3 t} \leq c_2 \cdot \Vert a \Vert e^{c_3 T} \leq r,
\end{array}
\]
and since $x^p(\cdot, S,a)_0 = x^0(\cdot,S,a)_0 = \phi$ we deduce from the previous inequality that we have 
$$\forall t \in (0,T], \Vert x^p(\cdot,S,a)_t - x^0(\cdot,S,a)_t \Vert \leq r.$$
And so $(t, x^p(\cdot,S,a)_t, u(t,S,a)) \in {\mathfrak B}$, and using the assumption of induction and the mean value theorem with the boundedness of $D_2f$ we obtain, for all $t \in [0,T]$, 
\[
\begin{array}{cl} 
\null & \Vert f(t, x^p(\cdot,S,a)_t, u(t,S,a)) - f(t,x^{p-1}(\cdot,S,a)_t, u(t,S,a)) \Vert \\
\leq & c_3 \cdot \Vert x^p(\cdot,S,a)_t -x^{p-1}(\cdot,S,a)_t \Vert \leq  c_3 \frac{(c_3 t)^{p-1}}{(p-1)!} c_2 \cdot \Vert a \Vert =  \frac{c_3^p t^{p-1}}{(p-1)!} c_2 \cdot \Vert a \Vert
\end{array}
\]
which implies 
\[
\begin{array}{cl} 
\null & \Vert x^{p+1}(t,S,a) - x^p(t,S,a) \Vert\\
 = &\Vert \int_0^t f(s, x^{p}(\cdot,S,a)_s, u(s,S,a)) ds -\int_0^t f(s, x^{p-1}(\cdot,S,a)_s, u(s,S,a)) ds \Vert\\
 \leq & \int_0^t \Vert  f(s, x^{p}(\cdot,S,a)_s, u(s,S,a)) - f(s, x^{p-1}(\cdot,S,a)_s, u(s,S,a))  \Vert ds\\
\leq & \int_0^t (c_2 \cdot  \Vert a \Vert \frac{c_3^p s^{p-1}}{(p-1)!}) ds =  c_2 \cdot \Vert a \Vert \frac{c_3^p}{(p-1)!} \int_0^t  s^{p-1} ds = c_2 \cdot \Vert a \Vert c_3^p \frac{t^p}{p!}
\end{array}
\]
and since $x^{p+1}(\cdot, S,a)_0 = x^p(\cdot,S,a)_0 = \phi$ we deduce from the previous inequality that we have finished the induction and so the formula \eqref{eq57} is proven.
\vskip1mm
From \eqref{eq57} it is easy to see that $([t \mapsto x^m(\cdot,S,a)_t])_{m \in \N}$ is a Cauchy sequence into $(C^0([0,T], C^0([-r,0], \R^n)), \Vert \cdot \Vert_{\infty})$ and consequently we obtain that $(x^m(\cdot,S,a))_{m \in \N}$ is a Cauchy sequence into $C^0([-r,T], \R^n)$. Since this last space is complete, there exists $x_*(\cdot,S,a) \in  C^0([-r,T], \R^n)$ such that $\lim_{m \rightarrow + \infty} \sup_{t \in [-r,T]} \Vert x^m(t,S,a) - x_*(t,S,a) \Vert = 0$ which implies that $\lim_{m \rightarrow + \infty} \sup_{t \in [0,T]} \Vert x^m(\cdot,S,a)_t - x_*(\cdot,S,a)_t \Vert = 0$. Using Lemma \ref{lem53} with $F = N_{u(\cdot,S,a)}$ and $\Gamma(t, \phi) = f(t,\phi, u(t,S,a))$ we obtain $\lim_{m \rightarrow + \infty} \sup_{t \in [0,T]} \Vert f(t, x^m(\cdot,S,a)_t, u(t,S,a)) - f(t, x_*(\cdot,S,a)_t, u(t,S,a)) \Vert = 0$ which implies, $\forall t \in [0,T]$, 
$$\lim_{m \rightarrow + \infty} \int_0^t f(s, x^m(\cdot,S,a)_s, u(s,S,a)) ds = \int_0^t f(s, x_*(\cdot,S,a)_s, u(s,S,a) ds.$$
Taking $m \rightarrow + \infty$ into the formula 
$$x^m(t,S,a) = \phi(0) + \int_0^t(f(s, x^{m-1}(\cdot,S,a)_s, u(s,S,a)) ds$$
we obtain $x_*(t,S,a) = \phi(0) + \int_0^t f(s, x_*(\cdot,S,a)_s, u(s,S,a)) ds$. Note that the set of the discontinuity points of the integrand of this last integral is included into $N_{u(\cdot,S,a)} = N_{\overline{u}} \cup \{ t_i + b_i : i \in \{ 1,...,N \}\} \cup  \{ t_i + b_i+ a_i : i \in \{ 1,...,N \}\}$. And so $x_*(\cdot,S,a) \in C^0([-r,T], \R^n) \cap PC^1([0,T], \R^n)$. From the last integral equation we deduce
$$\forall t \in [0,T] \setminus N_{u(\cdot,S,a)}, \frac{\partial x_*(t,S,a)}{\partial t} = f(t, x_*(\cdot,S,a)_t, u(t,S,a)).$$
Since $x^m(\cdot,S,a)_0 = \phi$ for all $m \in \N$, we have also $x_*(\cdot,S,a)_0 = \phi$. Then using the uniqueness of the solution of a Cauchy problem we obtain that 
\begin{equation}\label{eq58}
x_*(\cdot,S,a) = x(\cdot,S,a).
\end{equation}
We have yet seen that $\Vert x^m(t,S,a) - x^0(t,S,a) \Vert = \Vert x^m(t,S,a) - \overline{x}(t) \Vert \leq c_2 \cdot \Vert a \Vert e^{c_3 T}$ for all $t \in [0,T]$. And so setting $c_4 := k e^{c_2 T}$ and taking $m \rightarrow + \infty$ we obtain $\Vert x(t,S,a) - \overline{x}(t) \Vert = \Vert x_*(t,S,a) -  \overline{x}(t) \Vert \leq c_4$.
\end{proof}
\begin{lemma}\label{lem55}
For all $S \in {\mathfrak S}$, the two assertions hold.
\begin{enumerate}
\item[(i)] $x(T,S,\cdot)$ is differentiable at $0$.
\item[(ii)] For all $i \in \{1,...,N \}$, 
$$\frac{\partial x(T,S,0)}{\partial a_i} = X(T,t_i)[f(t_i,\overline{x}_{t_i}, v_i) -  f(t_i,\overline{x}_{t_i},\overline{u}(t_i))].$$
\end{enumerate}
\end{lemma}
\begin{proof}
After Lemma \ref{lem51} we know that $z(T,S,\cdot)$ is differentiable at 0. Then to prove (i) it suffices to prove that the mapping $x(T,S,\cdot) - z(T,S,\cdot)$ is differentiable at 0. We introduce 
\begin{equation}\label{eq59}
\zeta(t,S,a) := (x(t,S,a) - z(t,S,a) - (x(t,S,0) - z(t,S,0)) = x(t,S,a) - \overline{x}(t) - z(t,S,a).
\end{equation}
and 
\begin{equation}\label{eq510}
\xi(t,S,a) := \frac{\partial \zeta(t,S,a)}{\partial t} - D_2f(t, \overline{x}_t, \overline{u}(t)) \zeta(\cdot,S,a)_t.
\end{equation}
We calculate
\[
\begin{array}{ccl}
\frac{\partial \zeta(t,S,a)}{\partial t}  &=& \frac{\partial x(t,S,a)}{\partial t}  - \overline{x}'(t)  -  \frac{\partial z(t,S,a)}{\partial t} \\
\null &=& f(t, x(\cdot,S,a)_t, u(t,S,a)) - f(t, \overline{x}_t, \overline{u}(t)\\
\null & \null & -D_2f(t, \overline{x}_t, \overline{u}(t))z(\cdot,S,a)_t - f(t, \overline{x}_t,u(t,S,a)) + f(t, \overline{x}_t, \overline{u}(t))\\
\null &=& f(t, x(\cdot,S,a)_t, u(t,S,a)) - f(t, \overline{x}_t,u(t,S,a))\\
\null &\null & -D_2f(t, \overline{x}_t, \overline{u}(t))z(\cdot,S,a)_t
\end{array}
\]
which implies, using \eqref{eq59} and \eqref{eq510}, that
\[
\begin{array}{ccl}
\xi(t,S,a) &=& \frac{\partial \zeta(t,S,a)}{\partial t} - D_2f(t, \overline{x}_t, \overline{u}(t))(x(\cdot,S,a) - \overline{x}_t - z(\cdot,S,a))\\
\null &= & f(t, x(\cdot,S,a)_t, u(t,S,a)) - f(t, \overline{x}_t,u(t,S,a))\\
\null & \null & - D_2f(t, \overline{x}_t, \overline{u}(t))(x(\cdot,S,a) - \overline{x}_t)\\
\null & = & \int_0^1 D_2 f(t, \overline{x}_t + \theta[x(\cdot,S,a) - \overline{x}_t], u(t,S,a)) d \theta (x(\cdot,S,a) - \overline{x}_t)\\
\null & \null &  - D_2f(t, \overline{x}_t, \overline{u}(t))(x(.,S,a) - \overline{x}_t)
\end{array}
\]
and so we obtain
\begin{equation}\label{eq511}
\xi(t,S,a) = \Vert x(\cdot,S,a)_t - \overline{x}_t \Vert_{\infty} \cdot E(t,S,a)
\end{equation}
where
$E(t,S,a) := \left(  \int_0^1 D_2 f (t, \overline{x}_t + \theta[x(\cdot,S,a) - \overline{x}_t], u(t,S,a)) d \theta  - D_2f(t, \overline{x}_t, \overline{u}(t)) \right)$ \\
$ \frac{ x(\cdot,S,a)_t - \overline{x}_t}{\Vert  x(\cdot,S,a)_t - \overline{x}_t  \Vert}$ if $x(\cdot,S,a)_t \neq \overline{x}_t$ and $E(t,S,a) := 0$ if $x(\cdot,S,a)_t = \overline{x}_t$.
Since $D_2f$ is bounded, there exists $c_4 > 0$ such that, for all $t \in [0,T]$, for all $a \in \overline{B}(0,\delta_2)$, $\Vert E(t,S,a) \Vert \leq c_4$. Note that $E(t,S,a) = 0$ when $t \in [0,T] \setminus \cup_{1 \leq i \leq N} I_i$, and consequently we have 
\[
\begin{array}{ccl}
\int_0^T \Vert E(t,S,a) \Vert dt &=& \int_{ \cup_{1 \leq i \leq N} I_i}  \Vert E(t,S,a) \Vert dt + \int_{[0,T] \setminus \cup_{1 \leq i \leq N} I_i}  \Vert E(t,S,a) \Vert dt\\
\null & =&  \sum_{i=1}^N \int_{I_i} \Vert E(t,S,a) \Vert dt = \sum_{i=1}^N \int_{t_i + b_i}^{t_i + b_i + a_i} \Vert E(t,S,a) \Vert dt\\
\null & \leq & \sum_{i=1}^N a_i c_4 = c_4 \cdot \Vert a \Vert
\end{array}
\]
that implies 
\begin{equation}\label{eq512}
\lim_{a \rightarrow 0} \int_0^T \Vert E(t,S,a) \Vert dt = 0.
\end{equation}
From \eqref{eq510} we have $\frac{\partial \zeta(t,S,a)}{\partial t} = D_2f(t, \overline{x}_t, \overline{u}(t)) \zeta(\cdot,S,a)_t + \xi(t,S,a)$, and from \eqref{eq59} we also have $\zeta(\cdot,S,a)_0 = x(\cdot,S,a)_0 - \overline{x}_0 - z(\cdot,S,a) = \phi - \phi - 0$, and so we can use Theorem \ref{th34} and assert that $\zeta(t,S,a) = X(t,0)0 + \int_0^t X(t,y)(\int_{-r}^{-y} d_2 \eta(y, \theta) 0(y + \theta))dy) + \int_0^t X(t, \alpha) \xi(\alpha, S,a) d \alpha  = 0+0 +  \int_0^t X(t, \alpha) \xi(\alpha, S,a) d \alpha$, and so we have established the following formula
\begin{equation}\label{eq513}
\zeta(T,S,a) = \int_0^T X(T,t) \xi(t,S,a) dt.
\end{equation}
We know that $X$ is bounded on $[0,T] \times [0,T]$, and using Lemma \ref{lem54}, \eqref{eq513} and \eqref{eq511} we obtain
\[ 
\begin{array}{l}
\Vert (x(T,S,a) - z(T,S,a) - (x(T,S,0) - z(T,S,0)) \Vert = \Vert \zeta(T,S,a) \Vert\\
 \leq \int_0^T (\Vert X(T,t) \Vert \cdot \Vert \xi(t,S,a) \Vert) dt \leq
\Vert X \Vert_{\infty} \int_0^T (\Vert x(\cdot,S,a)_t - \overline{x}_t \Vert_{\infty} \cdot \Vert E(t,S,a) \Vert )dt\\
\leq \Vert X \Vert_{\infty} \cdot c_3 \cdot \Vert a \Vert \cdot \int_0^T \Vert E(t,S,a) \Vert dt 
\end{array}
\]
which implies (using \eqref{eq512}) that the mapping $x(T,S,\cdot) - z(T,S,\cdot)$ is differentiable at 0 and that its differential at 0 is equal to zero. Then using lemma \ref{lem51} we obtain $D_a X(T,S,0)a = D_a z(T,S,0) a = \sum_{i=1}^N a_i X(T,t_i) (f(t_i, \overline{x}_{t_i}, v_i) - f(t_i, \overline{x}_{t_i},\overline{u}(t_i)))$.
\end{proof}
The following multiplier rule comes from \cite{Mi}.
\begin{lemma}\label{lem56}
Let ${\mathcal V}$ be a neighborhood of 0 into $\R^N$. Let $n_i, q_e \in \N$. Let $\psi^j : {\mathcal V} \rightarrow \R$ be differentiable functions at 0, for $j \in \{ 0,...,n_i + n_e \}$. We assume that 0 is a solution of the following maximization problem
\[
\left\{
\begin{array}{rl}
{\rm Maximize} & \psi^0(a)\\
{\rm when} & a \in {\mathcal V} \cap \R^n_+\\
\null & \forall j = 1,...,n_i, \; \psi^j(a) \geq 0\\
\null & \forall j = n_i+1, ..., n_i+n_e, _; \psi^j(a) = 0.
\end{array}
\right.
\]
Then there exists $(\lambda_0,..., \lambda_{n_i + n_e}) \in \R^{1+ n_i + n_e}$ which satisfies the following conditions.
\begin{enumerate}
\item[(i)] $(\lambda_0,..., \lambda_{n_i + n_e})$ is non zero
\item[(ii)] $\forall j = 0,...,n_i, \; \lambda_j \geq 0$
\item[(iii)] $\forall j = 1,...,n_i, \; \lambda_j \psi^j(0) = 0$
\item[(iv)] $\forall a \in \R^N_+, \; \sum_{j= 0}^{n_i + n_e} \lambda_j D\psi^j(0).a \leq 0$.
\end{enumerate}
\end{lemma}
Now we apply this multiplier rule to our problem that permits us to obtain the following lemma.
\begin{lemma}\label{lem57}
Let $S \in {\mathfrak S}$. Then there exists $\Lambda(S) = (\lambda_j)_{0 \leq j \leq n_i + n_e} \in \R^{1+ p+ q}$ which satisfies the following properties.
\begin{enumerate}
\item[(i)] $\sum_{j=0}^{n_i + n_e} \vert \lambda_j \vert = 1$
\item[(ii)] $\forall j = 0,...,n_i, \; \lambda_j \geq 0$
\item[(iii)] $\forall j = 1,...,n_i, \; \lambda_j g^j(\overline{x}(T)) = 0$
\item[(iv)] $\forall i = 1,...,N$, \\ $\sum_{j=0}^{n_i + n_e} \lambda_j Dg^j(\overline{x}(T))X(T, t_i)[f(t_i, \overline{x}_{t_i}, v_i) -f(t_i, \overline{x}_{t_i}, \overline{u}(t_i))] \leq 0 $.
\end{enumerate}
\end{lemma}
\begin{proof}
Since the process $(\overline{x}, \overline{u})$ is optimal for $({\mathfrak M})$, it is also optimal among the processes $(x(\cdot,S,a), u(\cdot,S,a))$ when $a$ belongs to a neighborhood of 0 into $\R^N_+$. Recall that $x(\cdot,S,0) = \overline{x}$ and $u(\cdot,S,0) = \overline{u}$. And then 0 is an optimal solution of the following maximisation static problem
\[
\left\{
\begin{array}{rl}
{\rm Maximize} & g^0(x(T,S,a))\\
{\rm when} & a \in \overline{B}(0, \delta_2) \cap \R^n_+\\
\null & \forall j = 1,...,n_i, \; g^j(x(T,S,a)) \geq 0\\
\null & \forall j = n_i+1, ..., n_i +n_e, _; g^j(x(T,S,a) )= 0.
\end{array}
\right.
\]
We use the previous lemma by setting $\psi^j(a) := g^j(x(T,S,a))$. Since the set of the lists of multipliers is a cone the non nullity of $(\lambda_j)_{0 \leq j \leq n_i + n_e}$ permits us to choose it to satisfy (i). The conclusions (ii) and (iii) are given by the previous lemma in a straightforward way. To treat the last condition, note that $D \psi^j(0) = Dg^j(\overline{x}(T))D_a x(T,S,0)$, and then the last conclusion of the previous lemma is: $\forall a \in \R^N_+$, $\sum_{j=0}^{n_i + n_e} \lambda_j Dg^j(\overline{x}(T)).D_a x(T,S,0)a \leq 0$. If $(e_i)_{1 \leq i \leq n}$ denotes the canonical basis of $\R^N$, we obtain 
$$0 \geq \sum_{j=0}^{n_i + n_e} \lambda_j Dg^j(\overline{x}(T))D_a x(T,S,0)a = \sum_{j=0}^{p+q} a_j \lambda_j Dg^j(\overline{x}(T))\frac{\partial x(T,S,0)}{\partial a_i}$$
and we conclude by using Lemma \ref{lem55}.
\end{proof}
The following lemma ensures the existence of multipliers which do not depend of $S \in {\mathfrak S}$.
\begin{lemma}\label{lem58}
There exists $(\lambda_j)_{0 \leq j \leq n_i + n_e} \in \R^{1+ n_i + n_e}$ which satisfies the following properties.
\begin{enumerate}
\item[(i)] $\sum_{j=0}^{n_i + n_e} \vert \lambda_j \vert = 1$
\item[(ii)] $\forall j = 0,...,n_i, \; \lambda_j \geq 0$
\item[(iii)] $\forall j = 1,...,n_i, \; \lambda_j g^j(\overline{x}(T)) = 0$
\item[(iv)] $\forall t \in [0,T]$, $\forall u \in U$, \\
$\sum_{j=0}^{n_i + n_e} \lambda_j Dg^j(\overline{x}(T))X(T,t)[f(t, \overline{x}_t, u) -f(t, \overline{x}_t, \overline{u}(t))] \leq 0$.
\end{enumerate}
\end{lemma}
The proof of this lemma is completely similar to this one of Lemme 3 in \cite{Mi}. The idea of this proof is the following one: at each $S \in {\mathfrak S}$ we associate the set $K(S)$ as the set of the $(\lambda_j)_{0 \leq j \leq n_i + n_e} \in \R^{1+ n_i + n_e}$ which satisfy the conclusion (i-iv) of Lemma \ref{lem57}. Denoting by ${\mathcal S}(0,1)$ the unit spere of  $\R^{1+ n_i + n_e}$ for the norm $\Vert (\lambda_j)_{0 \leq j \leq n_i + n_e} \Vert := \sum_{i=0}^{n_i + n_e} \vert \lambda_j \vert$, we see that $K(S)$ is nonempty (by Lemma \ref{lem57}), is closed into the compact ${\mathcal S}(0,1)$. For all finite list $(S_k)_{1 \leq k \leq {\ell}}$ of elements of ${\mathfrak S}$, we can build $S^* \in {\mathfrak S}$ such that $K(S^*) \subset \cap_{1 \leq k \leq {\ell}} K(S_k)$ that proves that  $\cap_{1 \leq k \leq {\ell}} K(S_k) \neq \emptyset$. Then the compactness of ${\mathcal S}(0,1)$ and the closedness of the $K(S)$ imply that $\cap_{S \in {\mathfrak S}}K(S) \neq \emptyset$. It suffices to take $(\lambda_j)_{0 \leq j \leq n_i + n_e}  \in \cap_{S \in {\mathfrak S}}K(S)$ and to note that each $(t,u) \in [0,T] \times U$ belongs to ${\mathfrak S}$ to obtain the conclusion (iv) of the lemma.
\vskip1mm
\noindent
{\bf The end of the proof of Theorem \ref{th41}.} The scalar multipliers $\lambda_0$,..., $\lambda_{n_i + n_e}$ are provided by Lemma \ref{lem58}. From lemma \ref{lem58} we see that the conditions (NN), (Si) and (Sl) of Theorem \ref{th41} are fulfilled.We define the function $p : [0,T] \rightarrow \R^{n*}$ by setting
\begin{equation}\label{eq514}
p(t) := \sum_{j=0}^{n_i + n_e} \lambda_j Dg^j(\overline{x}(T))X(T,t).
\end{equation}
From (iv) of Lemma \ref{lem58}, we see that the condition (MP) is fulfilled. Since $X(T,T) = I$ (identity), from \eqref{eq514} we see that the condition (T) is fulfilled. 
Using Proposition \ref{prop35}, we have $X(T,t) = I - \int_t^T X(T, \alpha) \eta^1(\alpha, t - \alpha) d \alpha$, which implies 
\[
\begin{array}{ccl}
p(t)& = &\sum_{j=0}^{n_i + n_e} \lambda_j Dg^j(\overline{x}(T))X(T,t) \\
\null & =&  \sum_{j=0}^{n_i + n_e} \lambda_j Dg^j(\overline{x}(T)) - \int_t^T \sum_{j=0}^{n_i + n_e} \lambda_j Dg^j(\overline{x}(T)) X(T, \alpha) \eta(\alpha, t - \alpha) d \alpha\\
\null & =& \sum_{j=0}^{n_i + n_e} \lambda_j Dg^j(\overline{x}(T)) - \int_t^T p(\alpha) \eta^1(\alpha, t - \alpha) d \alpha
\end{array}
\]
and so we see that the function $[t \mapsto p(t) + \int_t^T p(\alpha) \eta(\alpha, t - \alpha) d \alpha]$ is constant on $[0,T]$. Note that $\alpha > t+r$ implies $t- \alpha < -r$ and then $\eta^1(\alpha, t- \alpha) = 0$ and consequently $\int_t^T p(\alpha) \eta^1(\alpha, t - \alpha) d \alpha = \int_t^{ \min \{t+r,T \} } p(\alpha) \eta^1(\alpha, t - \alpha) d \alpha =\int_t^{ \min \{t+r,T \} } p(\alpha) \eta(\alpha, t - \alpha) d \alpha$ that proves that the condition (AE) holds.
\vskip1mm 
Using (QC), if $p(T) = 0$ since (NN), (Si) and (Sl) hold we obtain a contradiction, and so we have $p(T) \neq 0$ under (QC). Since $X(T,t) = I - \int_t^T R(\xi, t) d \xi$, and since $R$ is bounded, if we choose $\epsilon \in (0, \Vert R \Vert_{\infty}^{-1}]$, then, when $t \in (T- \epsilon, T]$, we obtain that $X(T,t)$ is invertible. Using (TC), we have $p(t) = p(T) X(T,t)$, ans since $p(T) \neq 0$, we obtain $p(t) \neq 0$. And so (A1) is proven. To prove (A2), we proceed by contradiction. We assume that there exists $\tau \in [0,T]$ such that $p(t) = 0$ when $t \in [\tau, \min\{ \tau + r, T \}]$. Note that $\tau < T$ since $p(T) \neq 0$. Using (AE) and (TC) we have, $p(\tau) + \int_{\tau}^{\min \{ \tau + r, T \}} p(\xi) \eta(\xi, \tau - \xi) d \xi = p(T)$ which implies $0 + 0 = p(T)$ that is impossible. And so (A2) is proven and the proof of Theorem \ref{th41} is complete.
\end{document}